\documentclass[11pt,reqno]{amsart} % please use amsart at 11pt
\usepackage[height=190mm,width=130mm]{geometry}
\usepackage{amsmath}
\usepackage{amssymb,latexsym}
\usepackage{amsrefs}
\usepackage[draft=true]{hyperref}
\usepackage{cite}
\usepackage{a4wide}
\usepackage{amscd}
\usepackage{graphics}
\usepackage{color}
\usepackage[all,cmtip]{xy}
\usepackage{tikz-cd}

\theoremstyle{plain}

\newtheorem{lemma}{Lemma}
\newtheorem{corollary}{Corollary}
\newtheorem{proposition}{Proposition}
\newtheorem{example}{Example}

\theoremstyle{definition}
\newtheorem{defn}{Definition}

\theoremstyle{remark}
\newtheorem{remark}{Remark}

\newcommand{\Z}{\ensuremath{\mathbb{Z}}}   %\Z integers Z
\newcommand{\N}{\ensuremath{\mathbb{N}}}
\newcommand{\Q}{\ensuremath{\mathbb{Q}}}

\newcommand{\tspp}{\ensuremath{\mathfrak{Pr}}}
\newcommand{\tsi}{\ensuremath{\mathfrak{In}}}
\newcommand{\injd}{\ensuremath{\mathfrak{In ^{-1}}}}
\newcommand{\spp}{\ensuremath{\mathfrak{\underline{Pr} ^{-1}}}}
\newcommand{\si}{\ensuremath{\mathfrak{\underline{In}^{-1}}}}

\newcommand{\proj}{\ensuremath{{\mathcal{P}_0}}}
\newcommand{\inj}{\ensuremath{{\mathcal{I}_0}}}
\newcommand{\Hom}{\operatorname{Hom}}

\newcommand{\Mod}{\operatorname{Mod-}R}
\newcommand{\Mods}{\operatorname{Mod-}S}
\newcommand{\sip}{\ensuremath{\mathfrak{siP}}}
\newcommand{\sppp}{\ensuremath{\mathfrak{sP}}}

\numberwithin{equation}{section} % to get equations numbered
\begin{document}

\title[Extension-reflecting modules]{Subinjective and Subprojective Extension-reflecting modules}

\author{ENG\.{I}N B\"uy\"uka\c{s}{\i}k}

\address{Izmir Institute of Technology \\ Department of Mathematics\\ 35430 \\ Urla, \.{I}zmir\\ Turkey}

\email{enginbuyukasik@iyte.edu.tr}



\begin{abstract}

Given a right \( R \)-module \( M \) and any short exact sequence of right \( R \)-modules
\[
0 \to A \to B \to C \to 0,
\]
it is well known that if both \( A \) and \( C \) belong to the subinjectivity domain \( \si(M) \) (resp., the subprojectivity domain \( \spp(M) \)) of \( M \), then \( B \) also belongs to the corresponding domain. Module classes satisfying this closure property are said to be closed under extensions.

\noindent Let $ \tsi(M) = \{ N \in \Mod \mid M \in \si(N) \}\,\, \text{and}  \,\, \tspp(M) = \{ N \in \Mod \mid M \in \spp(N) \}.$
Unlike \( \si(M) \) and \( \spp(M) \), the classes \( \tsi(M) \) and \( \tspp(M) \) are not, in general, closed under extensions.

In this paper, we investigate certain classes of modules and rings that ensure the extension-closure of \( \tsi(M) \) and \( \tspp(M) \).
We prove that if  the injective hull of $M$ is projective, then $\tsi(M)$ is closed under extension; and  if $M$ is a homomorphic image of a module which is both projective and injective, then $\tspp(M)$ is closed under extensions. As a consequence, over a QF-ring, the classes 
$\tsi(M)$ and $\tspp(M)$ are closed under extensions for every module $M$.
We further explore several implications of these results and present examples of modules with the above properties over arbitrary rings.

Additionally, we introduce the rings whose right modules of finite length are homomorphic images of injective modules.  
Among other results, we prove that, if $\tsi(R)$ is closed under extensions, then modules of finite length are homomorphic image of injectives if and and only if  simple modules are homomorphic image of injectives.

\end{abstract}

\maketitle

\section{introduction}

Throughout this paper, $R$ will be an associative ring with identity and by a module we shall mean  a unital right $R$-module, unless otherwise stated.  The category of all right $R$-modules over a ring $R$ will be denoted by $\Mod$. For a module $M$, we write $N \leq M$ to indicate that $N$ is a submodule of $M$. The classes of injective and projective modules are denoted by $\inj$ and $\proj$, respectively.

Recall that, given two modules $M$ and $N$,  \( M \) is \( N \)-injective if for any submodule \( K \) of \( N \), any \( R \)-homomorphism \( f : K \to M \) extends to some member of $\Hom_R(N, M)$. The class of all right \( R \)-modules \( N \) such that \( M \) is injective relative to \( N \), denoted by \( \injd(M) \), is called the \emph{injectivity domain} of \( M \). The class of semisimple modules is contained in  $\injd(M)$ for any module module $M$. The modules whose injectivity domain consists only semisimple modules are said to be poor (see,  \cite{AAL}). Poor modules are the opposites of injectives in the sense that their injectivity domain is as small as possible.  For any module $N$, $\injd (N)$ is closed under taking submodules, finite direct sum and quotients (see, \cite{AF}). That is, $\injd (N)$ is a hereditary pretorsion class. For more detailed and comprehensive study of injectivity domains, we refer to  \cite{AF} and \cite{SL}.

Another alternative approach to the usual notions of injectivity and projectivity of modules has been investigated in \cite{AL} and \cite{HLM}, respectively. Since then, these topics have attracted considerable attention from researchers and have come to occupy an important place in module and ring theory. Many well-known module and ring structures have been characterized in terms of these new notions.

Following \cite{AL}, given two modules $A$ and $B$, $B$ is said to be $A$-subinjective if for every extension $C$ of $A$ and every homomorphism $\phi: A \to B$ there exists a homomorphism $\varphi : C \to B$ such that $\varphi |_A= \phi$. For a module $B$, the subinjectivity domain of $B$ is defined to be the collection of all modules $A$ such that $B$  is $A$-subinjective. Subprojectivity and subprojectivity domains are defined dually (see \cite{HLM}). Explicitly, the subinjectivity and subprojectivity domains of a module $B$, respectively,  are the classes: $\si (B)= \{ A \in \Mod \mid \text{$B$ is  $A$-subinjective}   \} \quad \text{and} \quad \spp (B)= \{ A \in \Mod \mid \text{$B$ is  $A$-subprojective}   \}.$
A right $R$-module $M$ is indigent (resp. p-indigent) if $\si(M)=\inj$ (resp. \spp(M)=\proj).
A ring $R$ is said to have no right subinjective (resp. subprojective) middle  class if every module is indigent (resp. p-indigent) or injective (resp. projective). 

Indigent and p-indigent modules, along with some of their generalizations, continue to be an active area of research. Several open problems remain, such as the existence of indigent modules over an arbitrary ring and the complete characterization of rings that lack a subinjective (or subprojective) middle class. In recent years, generalizations of these concepts to different categories and classes of modules have also attracted significant attention, see for example \cite{AABGO, AD, D1, D2, D3}.

For a class of modules $\mathcal{A}$, the class  $\sip(\mathcal{A})=\{\si (M) \mid M \in \mathcal{A}   \}$ is said to be the subinjective profile of $\mathcal{A}.$ Suprojectivity profile of  $\mathcal{A}$ is defined similarly, and denoted as $\sppp(\mathcal{A})$.  The classes of simple modules, cyclic modules,  modules of finite length and finitely generated modules will be denoted by $\mathcal{S},\, \mathcal{C}$, \,$\mathcal{FL}$  and $\mathcal{FG}$, respectively.
%Every simple right $R$-module is cyclic and indecomposable, thus the following implications hold.
%$$\sip(\mathcal{S}) \subseteq \sip(\mathcal{C})\,\, \text{and}\,\, \sip(\mathcal{S}) \subseteq \sip(\mathcal{ID}).$$
Subinjective profiles of $\mathcal{A}$  extensively studied in the literature, especially for the classes $\mathcal{S}$ and $\Mod$  (see,  \cite{ABE,  AL, DO, Durgun, ADD, SA, SL}). Recently, $\sppp(\mathcal{FG})$  and $\sip(\mathcal{FG})$  are studied in \cite{DKS} and \cite{DO}, respectively. 

%The authors investigate the cases when $\sip(\mathcal{FG})$ has one or two elements.

%For a module $M$, let $\tsi (M)= \{ N \in \Mod \mid M \in \si(N)  \}$ and  $\tspp (M)= \{ N \in \Mod \mid M \in \spp(N)\}.$ It is clear that, $M$ is injective (resp. projective) if and only if $\tsi(M)=\{ \Mod \}$ (resp. $\tspp(M)=\{ \Mod \}$). 

In the study of rings and modules related to subinjectivity and subprojectivity domains, it is important to understand certain intrinsic properties of these domains. For example, both subinjectivity and subprojectivity domains are closed under finite direct sums and extensions (Lemma \ref{lem:subinjectivityclosedunderextension}). Moreover, $R$ is right hereditary if and only if  $\si(M)$ is closed under homomorphic images for each $M \in \Mod$  (\cite[Theorem 2.1]{ABD}) if and only if   $\spp(M)$ is closed submodules for each $M \in \Mod$ (\cite[Proposition 2.15]{HLM}).
As can be seen easily, $\si(M)$ is closed under submodules if and only if $M$ is injective; and hence $R$ is semisimple artinian if and only if $\si(N)$ closed under submodules for each $N \in \Mod.$ 

As demonstrated, subinjectivity domains possess some of the closure properties enjoyed by injectivity domains, but only when certain conditions on the module or the underlying ring are satisfied. This distinction motivates a deeper investigation into the structural conditions under which subinjectivity and subprojectivity domains enjoy additional certain closure properties.

To contribute to the studies in this direction, in this paper, we introduce and investigate the following notions. For a module $M$, let $\tsi (M)= \{ N \in \Mod \mid M \in \si(N)  \}$ and  $\tspp (M)= \{ N \in \Mod \mid M \in \spp(N)\}.$ It is clear that, $M$ is injective (resp. projective) if and only if $\tsi(M)=\{ \Mod \}$ (resp. $\tspp(M)=\{ \Mod \}$). 

Given a right $R$-module $M$ and any short exact sequence $0\to A \to B \to C \to 0$ of right $R$-modules, $M$ is said to be {\it subinjective extension-reflecting (si.e.r)}  in case $M \in \si(A)\cap \si(C)$ implies $M \in \si(B)$.  $M$ is {\it subprojective extension-reflecting (sp.e.r)} if $M \in \spp(A)\cap \spp(C)$ implies $M \in \spp(B)$. Equivalently that is to say, $M$ is si.e.r. (resp. sp.e.r.) if and only if $\tsi(M)$ (resp. $\tspp(M))$ is closed under extensions.

It is worth mentioning that, depending on the context they appear, the equivalent expressions "N is si.e.r (sp.e.r)" and "$\tsi(N)\, (\tspp(N))$ is closed under extensions" will be used interchangeably throughout the paper.

Injective and projective modules are trivial examples of si.e.r and sp.e.r modules, respectively. There are modules that are not si.e.r or sp.e.r (see Examples   \ref{example} and  \ref{examplenotsp}).

The paper is organized as follows.

In section 2, we first give some examples in order to show that, for a given module $N,$ $\tsi (N)$ and $\tspp (N)$ are not closed under extensions in general. We provide  sufficient conditions ensuring that $\tsi (M)$ and $\tspp (M)$  are closed under extensions. Namely, we prove that, (1) if $N$ is a submodule of a module that is both injective and projective, then $\tsi(N)$ is closed under extensions, and (2)  if $N$ is a homomorphic image of a module that is both injective and projective, then $\tspp(N)$ is closed under extensions. We conclude that, over a $QF$-ring every right (or left) module $N$, both $\tsi(N)$ and $\tspp (N)$ are closed under extensions. We also show that, over a ring $R$ without subinjective (subprojective) middle class, which is not necessarily $QF,$ $\tsi(N)$ (\tspp(N)) is closed under extensions for each $N \in \Mod$.

The class of si.e.r. modules is closed under taking submodules, whereas the class of sp.e.r. modules is closed under taking factor modules.
 For an si.e.r module $M$, we prove that $(1)$ $M \in \si(S) $ for each  $S \in \mathcal{S}$ if and only if $M \in \si (A)$ for each   $A \in \mathcal{FL}$; and $(2)$ $M \in \si(C) $ for each  $C \in \mathcal{C}$ if and only if $M \in \si (D)$ for each  $D \in \mathcal{FG}$. Dual of these results are also obtained for sp.e.r modules, as well. 

For a $QF$-ring, we prove that for each $\mathcal{A},\, \mathcal{B} \in \{\mathcal{S},\, \mathcal{C},\,\mathcal{FL},\,\mathcal{FG} \}$
$$\inj=\bigcap_{A \in \mathcal{A}}\si(A)= \bigcap_{B \in \mathcal{B}}\spp(B)=\proj.$$ %\si(A)=\sip(\mathcal{FL})=\sip(\mathcal{C})=\sip(\mathcal{FG})=\sppp(\mathcal{S})=\sppp(\mathcal{FL})=\sppp(\mathcal{C})=\sppp(\mathcal{FG})=\proj.$$

Since $\inj=\proj$ over a QF-ring, indigent and p-indigent modules coincide. We obtain that, for a QF-ring, a module $B$ is indigent if and only if $\si(A) \cap \si (B/A)=\inj$ for each $A \leq B.$

As a byproduct of these results, for several well-known module classes, such as, $FG$-injective, $C$-injective,  $FG$-projective and $C$-projective we obtain that: (a) If $M$ is si.e.r., then $M$ is $FG$-injective if and only if  it is $C$-injective. (b) If $M$ is sp.e.r., then $M$ is $FG$-projective if and only if it is $C$-projective. These results leads to: over a QF-ring, all the notions $FG$-injective, $C$-injective,  $FG$-projective and $C$-projective coincide with  $\inj$ (=$\proj$).

It is well known that a ring \(R\) is QF if and only if every right \(R\)-module is a homomorphic image of an injective module. In~\cite{meric}, the authors study rings whose proper cyclic modules are homomorphic images of injective modules. More recently, rings whose simple right modules are homomorphic images of injective modules have been investigated in~\cite{BL}; such rings are referred to as right dual Kasch rings.

In section 3, we investigate rings whose modules of finite (composition) length are homomorphic images of injective modules. Such rings are said to satisfy property~\((Q)\). 
We prove that a ring \(R\) satisfies \((Q)\) if and only if \(R \in \si(N)\) for each \(N \in \mathcal{FL}\). We show that right artinian rings satisfying \((Q)\) are QF. An example is provided to demonstrate that the class of dual Kasch rings properly contains the class of rings satisfying \((Q)\). When \(\tsi(R)\) is closed under extensions, we prove that \(R\) satisfies \((Q)\) if and only if \(R\) is dual Kasch. It is also shown that factor rings of rings satisfying \((Q)\) need not themselves satisfy \((Q)\). For a commutative noetherian ring \(R\), we prove that every factor ring of \(R\) satisfies \((Q)\) if and only if every factor ring of \(R\) is QF.

\section{Extension-reflecting modules}

%\begin{defn}  A class of right $R$-modules $\mathcal{D} $ is said to be closed under extension if for any short exact sequence of right $R$-modules $$0\to A \to B \to C \to 0,$$the condition $A,\,C \in \mathcal{D}$ implies $C \in \mathcal{D}.$
%\end{defn}

The following lemma is well known. From which we see that, subinjectivity domain of any right module are closed under extension. We include the proof for completeness.

\begin{lemma}\label{lem:subinjectivityclosedunderextension} For every module $N$,   $\si(N)$ and $\spp(N)$ are closed under extensions. \end{lemma}

\begin{proof}  First we prove that $\si(N)$ is closed under extensions. For a module $B$ and a submodule $A$ of $B$, we shall show that if $A$ and $B/A$ are contained in $\si(N),$ then $B \in \si(N).$  Let $f:B \to E$ be a homomorphism. Consider the following pushout diagram  of $f$ and $\pi,$ where $\pi :B \to B/A$ is the canonical epimorphism.

\[
\xymatrix{
& 0 \ar[d] & 0 \ar[d] & 0 \ar[d] & \\
0 \ar[r] & A \ar[r]^{\theta} \ar@{=}[d] & B \ar[r] \ar[d]^{f} & B/A \ar[r] \ar[d]^{g} & 0 \\
0 \ar[r] & A \ar[r]^{\alpha} & E \ar[r] \ar[d] & Y \ar[r] \ar[d] & 0 \\
& & Z \ar[r]^{=} \ar[d] & Z \ar[d] & \\
& & 0 & 0 &
}
\]

Applying $\Hom (-, N)$ to the diagram above we obtain the following commutative diagram.

\[
\xymatrix{
& 0 \ar[d] & 0 \ar[d] & \\
0 \ar[r] & \mathrm{Hom}(Z, N) \ar[r]^{\cong} \ar[d] & \mathrm{Hom}(Z, N) \ar[d] & \\
0 \ar[r] & \mathrm{Hom}(Y, N) \ar[r] \ar[d]^{g^*} & \mathrm{Hom}(E, N) \ar[r]^{\alpha ^*} \ar[d]^{f^*} & \mathrm{Hom}(A, N) \ar[r] \ar[d]^{\cong} & 0 \\
0 \ar[r] & \mathrm{Hom}(B/A, N) \ar[r]^{\pi^*} & \mathrm{Hom}(B, N) \ar[r]^{\theta^*} & \mathrm{Hom}(A, N) \ar[r] & 0
}
\]
Since $A \in  \si (N)$, the maps $\alpha ^*$ and $\theta ^*$ are epimorphisms. Thus, the second and the third rows are exact. On the other hand, as $B/ A \in \si (N),$ the map $g^*$ is an epimorphism, as well.  Hence $f^*$ is an epimorphism by the 5-Lemma, and so $B \in \si (N).$ Hence $\si(N)$ is close under extensions.

The proof for $\spp(N)$ is dual and can be seen  in \cite[Proposition 1]{AB}. 
\end{proof}

%In searching a variation of Lemma \ref{lem:subinjectivityclosedunderextension}, we need the following definition.

For a module $M$, let $\tsi (M)= \{ N \in \Mod \mid M \in \si(N)  \}$ and  $\tspp (M)= \{ N \in \Mod \mid M \in \spp(N)\}.$ It is clear that, $M$ is injective if and only if $\tsi(M)=\{ \Mod \}$.

\begin{defn} Let $M$ be a module and $\mathbb{E}:\, 0\to A \to B \to C \to 0$ be a short exact sequence in $\Mod$. We say that: 
\begin{enumerate}
\item $M$ is {\it  subinjective extension-reflecting (si.e.r) } if  $\tsi (M)$ is closed under extension. That is, the condition $M \in \si(A) \cap \si(C)$ implies $M \in \si (B)$ for every short exact sequence $\mathbb{E}.$
\bigskip
\item $M$ is  {\it  subprojective extension-reflecting (sp.e.r) }   if $\tspp (M)$  is closed under extension. That is, the condition $M \in \spp(A) \cap \spp(C)$ implies $M \in \spp(B)$ for every short exact sequence $\mathbb{E}.$

\end{enumerate}
\end{defn}

\begin{defn} $R$ is said to be right fully si.e.r. (resp. sp.e.r.) if every right $R$-module is si.e.r. (resp. sp.e.r.). Left fully si.e.r and left fully sp.e.r. rings are defined similarly.
\end{defn}

We begin by presenting examples to show  there exist modules which are not si.e.r or sp.e.r.

\begin{example}\label{example} Consider the ring $$R=\{\left(\begin{array}{cc}a & (b,c) \\0 & a\end{array}\right) \mid a,b,c \in \Z_2\}.$$Then $R$ is a commutative (finite) artinian local ring with $$J(R)= soc(R)=\left(\begin{array}{cc}0 & \Z_2 \times \Z_ 2 \\0 & 0\end{array}\right).$$ 
By \cite[Theorem 3.4]{BL} $R$ is dual Kasch. Thus $R \in \si(R/J(R))$ by \cite[Theorem 2.3]{BL}. Then  $R\in \si(J(R))$, because $J(R)$  is  semisimple. Hence $R \in \si(J(R)) \cap \si(R/J(R)).$   On the other hand, since $R$ is local and $soc(R)$ is not simple, $R$ is not a QF ring by \cite[Theorem 15.27]{Lammodules}, and so   $R$ is not self-injective. Thus $R \notin \si(R).$ Hence $R$ is not si.e.r.
\end{example}

\begin{example}\label{examplenotsp} In the category of $\Z$-modules, $\Q$ is not sp.e.r. Consider the exact sequnce $0\to \Z \to \Q \to \Q / \Z \to 0.$ Clearly, $\Z$ is $\Q$-subprojective. Also  $Hom (\Q / \Z,\, \Q)=0$ implies $\Q / \Z$ is $\Q$-subprojective. There is an epimorphism $\Z^{(\N)} \to \Q \to 0.$ Since $Hom (\Q, \Z^{(\N)})=0,$ the identity map $1_{\Q} :\Q \to \Q$ can not be lifted to a homomorphism $\Q \to \Z^{(\N)}.$ Thus $\Q$ is not $\Q$-subprojective. Summing up, we have $\Q \in \spp(\Z) \cap \spp(\Q /\Z)$ and $\Q \notin \spp(\Q).$ Hence $\Q$ is not sp.e.r.
\end{example}

\begin{example}\label{E1}\noindent
\begin{enumerate}
\item Every injective module is si.e.r., since $\inj \subseteq \si (N)$ for every  $N \in \Mod.$

\bigskip

\item If $M$ is an si.e.r module, then $M\oplus E$ is si.e.r for every injective module $E:$ For every  module $N$, $E \in \si(N)$ by $(1).$  Hence $M\oplus E\in \si(N)$ if and only if $M \in \si(N).$

\bigskip

\item Any right hereditary right noetherian ring $R$ is si.e.r: By \cite[Theorem 19]{ABE}, for such a ring we have $R\in \si (N)$ if and only if $N$ is injective. Since injective modules are closed under extension, the claim follows.

\item As a dual of $(1)$ and $(2)$, every projective module is sp.e.r., and if $M$ is an sp.e.r. module, then $M\oplus P$ is sp.e.r for every projective module $P.$

\end{enumerate}
\end{example}

\begin{proposition} The classes of si.e.r and sp.e.r modules are closed under finite direct sums.
\end{proposition}

\begin{proof}

Let $M=M_1 \oplus \cdots \oplus M_n$ be a direct sum of modules. For a module $N,$ $M \in \si(N)\,  (\spp(N))$  if and only if $M_i \in  \si(N) \, (\spp(N))$ for $i=1,\cdots,n$ (see \cite[Proposition 2.5]{AL} and \cite[Proposition 2.12]{HLM}) Thus, $M$ is si.e.r. (resp. sp.e.r.) if and only if $M_i$ is si.e.r. (resp. sp.e.r.) for each $i=1,\cdots,n.$ 
\end{proof}

Following \cite{ABE}, a module $M$ is said to be a test module for injectivity by subinjectivity or t.i.b.s. module if $M \in \si(N)$ for some module $N$, then $N$ is injective.

\begin{proposition}\label{prop:tibsser} Every t.i.b.s. module is si.e.r.
\end{proposition}

\begin{proof} Suppose $N$ is a t.i.b.s. module. Assume that, $N \in \si(A) \cap \si(B/A)$ for some module $B$ and $A\leq B.$ Then $A$ and $B/A$ are injective, because $N$ is a t.i.b.s. module. Then $B\simeq A \oplus B/A$ is injective, and so $N\in \si(B).$ Therefore $N$ is si.e.r.
\end{proof}

%\begin{proposition}

%\end{proposition}

Every ring has a t.i.b.s. module, which  must be noninjective, unless the ring is semisimple (see, \cite[proposition 1]{ABE}). This implies the following.

\begin{corollary} Every si.e.r module is injective if and only if $R$ is semisimple artinian.
\end{corollary}

A module $N$ is said to be indigent if $\si(N)=\inj.$ A ring $R$ is said to have no right subinjective middle class if every module is either injective or indigent.  $M$ is p-indigent if $\spp(M)=\proj.$ $R$ is a ring with no right subprojective middle class if every right $R$-module is p-indigent or projective.

\begin{proposition} Let $R$ be ring with no right subinjective middle class. Then $R$ is right fully si.e.r.
\end{proposition}

\begin{proof} Suppose $R$ has no right subinjective middle class. Then every module is t.i.b.s. or injective by \cite[Proposition 2]{ABE}.  Let $M$ be a module. If $M$ is a t.i.b.s module then the claim follows by Proposition \ref{prop:tibsser}. If $M$ is injective, then $M$ is si.e.r by Example \ref{E1}(1). 
\end{proof}

\begin{proposition}\label{prop:} If $R$ has no right subprojective middle class, then $R$  is right fully sp.e.r.
\end{proposition}

\begin{proof}  Suppose \( M \in \spp(A) \cap \spp(B/A) \). By hypothesis, \( A \) is either p-indigent or projective. In the first case, \( A \) is p-indigent, which implies that \( M \) is projective and hence \( M \in \spp(B) \). In the second case, \( A \) is projective. If \( B/A \) is p-indigent, then \( M \) is again projective, so \( M \in \spp(B) \). If \( B/A \) is projective, then \( B = A \oplus B/A \) is projective, which again implies \( M \in \spp(B) \).
\end{proof}

\begin{example}\label{example:nomc} The following rings have no right subinjective and  right subprojective middle classes, as shown in \cite[Examples 11–12]{ABE} and \cite[Theorem 3.1]{Durgun}, respectively. Hence,  these rings are both fully si.e.r and fully sp.e.r.

\begin{enumerate}
\bigskip
\item $R=\left(\begin{array}{cc}D & D \\0 & D\end{array}\right)$, $D$ is a division ring. 

\bigskip

\item $R$ is an artinian chain ring.
\end{enumerate}
\end{example}

%\begin{proposition} Let $R$ be a commutative domain. Then every projective $R$-module is subinjectively extension-reflecting if and only if $R$ is a local Dedekind domain (=DVR.)
%\end{proposition}

As illustrated by the Examples \ref{example} and \ref{examplenotsp}, for a given module $M$ neither $\tsi (M)$ nor $\tspp (M)$ is closed under extensions in general. The following two successive results provide sufficient conditions ensuring that $\tsi (M)$ and $\tspp (M)$  are closed under extensions.

\begin{proposition}\label{lem:injdm}  If  $M$ is a submodule of a module $E$ that is both injective and projective, then  $M$ is si.e.r.
\end{proposition}

\begin{proof} Let      $ \xymatrix{ 0 \ar[r] & A \ar[r]^{\alpha}  & B   \ar[r]^{\beta} & C \ar[r] & 0}  $ be a short exact sequence of modules such that $M\in \si(A) \cap \si(C).$ Without loss of generality, we may assume that, $A \leq B,$ $C=B/A$, $\alpha=e$ and $\beta=\pi$, where $e$ and $\pi$ are the inclusion and the canonical epimorphism, respectively.

Let \( f : M \to B \) be a homomorphism. We consider two cases based on the image of \( f \). Suppose first that \( f(M) \subseteq A \).  
Since \( M \in \si(A) \), there exists a homomorphism \( k : E \to A \) such that \( k  i = f \), where \( i : M \to E \) is the inclusion map. Thus, \( k \) extends \( f \).
Now suppose that \( f(M) \nsubseteq A \).  
Then \( \pi  f : M \to B/A \) is a nonzero homomorphism, where \( \pi : B \to B/A \) is the canonical epimorphism. Since \( M \in \si(B/A) \), there exists a homomorphism \( g : E \to B/A \) such that \( \pi  f = g  i \).

Now, since \( E \) is projective by assumption, there exists a homomorphism \( h : E \to B \) such that \( \pi  h = g \). Therefore,
\[
\pi  f = g  i = (\pi  h)  i = \pi  (h  i),
\]
which implies that \( \pi  (f - h  i) = 0 \), so \( (f - h  i)(M) \subseteq \ker(\pi) = A \).

Consider the homomorphism \( f - h  i : M \to A \). Since \( M \in \si(A) \), there exists a homomorphism \( t : E \to A\) such that \( f - h  i = t  i \). It follows that
\[
f = (h - t)  i,
\]
so \( h - t \) is a homomorphism from \( E \to B \) that extends \( f \).

\[
\xymatrix@R=1.5cm@C=1.8cm{
  & 0 \ar[r] & M \ar[r]^i \ar[d]_f &  E  \ar@{-->}[dl] ^h \ar@{-->}[d]^g  \ar@{-->}[dll]^t  \\
0 \ar[r] & A \ar[r] & B \ar[r]^\pi & B/A \ar[r]  & 0
}
\]
In either case, we have shown that any homomorphism \( f : M \to B \) extends to a homomorphism \( E \to B \), which implies \( M \in \si(B) \). This completes the proof.
\end{proof}

Dual to Proposition \ref{lem:injdm} we have the following.

\begin{proposition}\label{prop:sper} Let \( M \) be a module. If \( M \) is an epimorphic image of a module that is both projective and injective, then \( M \) is sp.e.r.
\end{proposition}

\begin{proof} Suppose $M\in \spp(A) \cap \spp (B/A).$ Let $g: P \to M$ be an epimorphism with $P$ is both projective and injective. Let $B$ be a module, $A \leq B$ and $f: B \to M$ a homomorphism. We need to show that, there is a homomorphism $k: B \to P$ such that $f=gk.$ 
 Let $i:A \to B$ and $\pi : B \to B/A$ be  the usual inclusion and the canonical epimorphism, respectively. By the hypothesis $M \in \spp (A),$ thus there is a homomorphism $h:   A \to P$ such that $fi=gh.$ By injectivity of $P,$ there is a homomorphism $t: B \to P$ such that, $ti=h.$ Then $fi=gti$, and so $(f-gt)i=0.$ That is, $A \subseteq ker (f-gt).$ 
 
 \[
\xymatrix@R=1.5cm@C=1.8cm{
0 \ar[r] & A \ar[r]^i \ar@{-->}[d]_h & B \ar@{-->}[dl]_t  \ar[r]^\pi \ar[d]_f & B/A \ar[r] \ar@{-->}[dll]_{\quad  \,\,\, u}& 0 \\
 & P  \ar[r]_g & M \ar[r] & 0
}
\]
By the factor theorem, there is a homomorphism $s: B/A \to M$ such that, $gt-f=s\pi .$ Now, as $M \in \spp (B/A)$ by the hypothesis, there is a homomorphism $u: B/A \to P$ such that, $gu=s.$ We obtain that, $$gu\pi= s \pi=gt-f \Rightarrow f=g(t-u\pi).$$Thus, the map $k$ we look for is $t-u \pi .$ Therefore, $M\in \spp (B).$ This completes the proof.
\end{proof}

The subsequent two consecutive results follow directly from Propositions \ref{lem:injdm} and \ref{prop:sper}.

\begin{corollary} The following hold for a module $M$.

\begin{enumerate}
\item  If $M$ is a submodule of a module that is both injective and projective, then every submodule of $M$ is si.e.r.

\item If \( M \) is an epimorphic image of a module that is both projective and injective, then every factor module of $M$ is sp.e.r.

\end{enumerate}
\end{corollary}

\begin{corollary} Let $R$ be a right self-injective ring. Then every right ideal is si.e.r, and every factor module of $R$ is sp.e.r.
\end{corollary}

%The following example shows that, there are modules which are not subinjectively extension-reflecting.

%\begin{example} Let $R$ be a local artinian ring of composition length 2 that is not QF. Then there is an exact sequence  $0 \subset S \to R  \to R/S \to 0,$ where $S$ is the unique simple right $R$-module up to isomorphism and $R/S \simeq S$ is simple. Since $R$ artinian and local, $R$ is a dual Kasch ring by \cite[Theorem 4.3.]{BL}. Thus $R \in \si(S) \cap \si(R/S).$ On the other hand, $R \notin \si(R).$ Otherwise, $R$ would be injective, which is not the case by the hypothesis. 
%\end{example}

%\begin{remark} Let $M$ be a finitely generated right $R$-module. If $E(M)$ is projective, then $E(M)$ is finitely generated. 
%\end{remark}

A ring $R$ is said to be QF-ring if $R$ is right noetherian and right self-injective.  $R$ is $QF$ if and only if $\inj =\proj$ (see, \cite{Lammodules}).
Hence we have the following.

\begin{corollary}\label{cor:injdomain}For a  QF-ring $R$, the following hold.
\begin{enumerate}
\item Every right  $R$-module  is si.e.r.
\item Every left  $R$-module  is si.e.r.
\item Every right  $R$-module  is sp.e.r.
\item Every left  $R$-module  is sp.e.r.
\end{enumerate}
\end{corollary}

\begin{corollary} Let $M$ be a module and $K \leq M$.  
\begin{enumerate}

\item  If $M$ is si.e.r, then  $$ M \in \si (K) \cap \si(M/K) \iff M  \in \inj $$

\item  If $M$ is sp.e.r, then  $$M \in \spp (K) \cap \spp(M/K) \iff  M \in \proj.$$

\end{enumerate}
\end{corollary}

\begin{proof} $(1)$ For the necessity, assume that $M \in \si (K) \cap \si(M/K)$. Then $M \in \si (M)$ by Proposition \ref{lem:injdm}. Hence $M$ is injective.
Conversely, suppose $M$ is injective. Then $M \in \si (N)$ for each $N \in \Mod$ Thus sufficiency follows.

$(2)$ Dual to the proof of  $(1).$
\end{proof}

%\begin{example}\label{example0} Consider the ring $$R=\{\left(\begin{array}{cc}a & (b,c) \\0 & a\end{array}\right) \mid a,b,c \in \Z_2\}.$$Then $R$ is a commutative (finite) local artinian ring with socle $soc(R)=\left(\begin{array}{cc}0 & \Z_2 \times \Z_ 2 \\0 & 0\end{array}\right).$ Since $R/soc(R) \simeq \Z_2$ and $(soc(R))^2=0$, we have $J(R)=soc(R)$. Thus $R$ is local ring whose unique maximal ideal is $soc(R).$ 
%Then $R \in \si(soc(R))\cap \si (R/soc(R))$ DETAY VERİLECEK. On the other hand, since $soc(R)$ is not simple, $R$ is not a QF ring by \cite{Lammodules}. In particular, $R$ is not self-injective. Therefore $R \notin \si(R).$
%\end{example}

\begin{proposition}\label{cor:ss=fl} Let $M$ be an si.e.r module. Then $$M \in \bigcap_{S \in \mathcal{S}} \si (S) \iff   M \in \bigcap_{A \in \mathcal{FL}} \si (A).$$  
\end{proposition}

\begin{proof} Suppose $M \in  \bigcap_{A \in \mathcal{FL}} \si (A).$ Since $\mathcal{S} \subseteq \mathcal{FL},$ we have $M \in \bigcap_{S \in \mathcal{S}} \si (S).$ To prove the reverse containment, suppose $M \in \bigcap_{S \in \mathcal{S}} \si (S).$ Let $0 \neq A \in \mathcal{FL}$ with composition length $n.$  Let $0 \subset S_1  \subset  \cdots \subset S_n =A$ be a composition series of $A.$ The proof is by induction on $n$. In the case $n=1,$ the module $A$ is simple, and so $M \in \si(A)$ by the assumption. Assume that $n >1$ and $M \in \si (B)$ for every module $B$ of composition length less than $n$.  Clearly,  $S_{n-1}$  and $S_n/S_{n-1}$ have composition length less than $n$. Thus $M \in \si(S_{n-1}) \cap \si (S_n /S_{n-1})$ by the induction hypothesis.  By  the hypothesis and Proposition \ref{lem:injdm}, $M$ is subinjectively extension-reflecting. Thus $M \in \si (S_n).$ Therefore $M \in \bigcap_{A \in \mathcal{FL}} \si (A)$ and the proof follows.
\end{proof}

\begin{proposition}\label{prop:C=FG}  Let $A$ be an si.e.r module. Then
%Let $R$ be a QF ring and $\mathcal{C}$ and $\mathcal{FG}$ be the classes of  cyclic and  finitely generated right $R$-modules, respectively. 

$$A \in \bigcap_{N \in \mathcal{C}} \si (N)   \iff   A \in \bigcap_{N \in \mathcal{FG}}\si (N).$$
\end{proposition}

%\begin{proof} Clearly  $\bigcap_{N \in \mathcal{FG}}\si (N) \subseteq \bigcap_{N \in \mathcal{C}} \si (N).$ To prove the reverse containtment, let $A \in \bigcap_{N \in \mathcal{C}} \si (N)$. Suppose $M$ is a finitely generated right $R$-module which can be generated by $k$ elements, say $m_1,\,m_2,\cdots , m_k.$ We shall prove that $A \in \si (M)$ by induction on $k.$ If $k=1$, then $M$ is cyclic and so $A \in \si (M)$ by the assumption. Assume that $k>1$ and $A \in \si (L)$ for every right $R$-module $L$ with a number of generators less than $k.$ Let $K=m_2 R + \cdots + m_{k}.$ Then $M /K\cong m_1R / K \cap m_1 R$ is a cyclic module, and thus  $A$ is  contained in $\si (K) \cap \si (M/K),$  by the induction hypothesis. Thus, $A \in \si (M)$ by Corollary \ref{cor:injdomain}. This proves that,  $\bigcap_{N \in \mathcal{FG}} \si (N) \subseteq \bigcap_{N \in \mathcal{C}} \si (N).$
%\end{proof}

\begin{proof} Clearly,
\[
\bigcap_{N \in \mathcal{FG}} \si(N) \subseteq \bigcap_{N \in \mathcal{C}} \si(N),
\]
since every cyclic module is finitely generated. To prove the reverse inclusion, suppose
\[
A \in \bigcap_{N \in \mathcal{C}} \si(N),
\]
i.e., \( A \in \si(N) \) for every cyclic right \( R \)-module \( N \). Let \( M \) be an arbitrary finitely generated module, and suppose it is generated by \( k \) elements \( m_1, m_2, \ldots, m_k \). We will prove by induction on \( k \) that \( A \in \si(M) \).  If \( k = 1 \), then \( M \) is cyclic and the claim \( A \in \si(M) \) follows by the assumption.
Assume \( k > 1 \), and suppose the claim holds for all finitely generated modules that can be generated by fewer than \( k \) elements. Consider the submodule
$K = m_2 R + \cdots + m_k R.$
Then the quotient module \( M/K \) is cyclic, since it is generated by the image of \( m_1 \) in \( M/K \). By assumption, \( A \in \si(M/K) \). Also, by the induction hypothesis (since \( K \) is generated by \( k - 1 \) elements), \( A \in \si(K) \).

By Corollary \ref{cor:injdomain}, it follows that
\[
A \in \si(K) \cap \si(M/K) \subseteq \si(M).
\]

Therefore, by induction, \( A \in \si(M) \) for every finitely generated module \( M \), which shows that
\[
A \in \bigcap_{N \in \mathcal{FG}} \si(N).
\]

Hence,
\[
\bigcap_{N \in \mathcal{C}} \si(N) \subseteq \bigcap_{N \in \mathcal{FG}} \si(N),
\]
completing the proof of equality.
\end{proof}

\begin{proposition}\label{prop:sp.e.r} Let $M$ be an sp.e.r module. Then

\begin{enumerate}

\item $M \in \bigcap_{S \in \mathcal{S}} \spp (S) \iff   M \in \bigcap_{A \in \mathcal{FL}} \spp (A).$
\bigskip
\item $M \in \bigcap_{N \in \mathcal{C}} \spp (N)   \iff   M \in \bigcap_{N \in \mathcal{FG}}\spp (N).$

\end{enumerate}

\end{proposition}

\begin{proof} The proofs are dual to the proof of   Proposition \ref{cor:ss=fl} and Proposition \ref{prop:C=FG}.
\end{proof}

The following result describe cyclic semisimple modules over semilocal rings.

\begin{proposition}\label{semilocal} Let $R$ be a semilocal ring. Then 

\begin{enumerate}
\item A semisimple module $M$ is cyclic if and only if it is a homomorphic image of $R/J(R).$

\item Any direct sum of nonisomorphic simple modules is cyclic.

\item If, in addition, $R$ is right si.e.r., then $\si(R/J(R))$ is the smallest element of  $\sip(\mathcal{FL}).$
\end{enumerate}
\end{proposition}

\begin{proof} $(1)$ Suppose $M$ is cyclic and semisimple. Let $f: R \to M$ be an epimorphism. Clearly, $R/ker(f)$ is semisimple, and so $ker(f)$ is an intersection of maximal right ideals. Then $J(R) \subseteq ker(f),$ and this leads to an epimorphism $g: R/J(R) \to M.$ This proves the necessity. Conversely, as $R/J(R)$ is cyclic and semisimple, any homomorphic image of it is cyclic and semisimple as well.

$(2)$ Let $M$ be a direct sum of nonisomorphic simple modules. Then by the semilocal assumption, $M$ is isomorphic to a direct summand of $R/J(R).$ Hence $M$ is cyclic by $(1).$

$(3)$ Since $R$ is semilocal, $R/J(R)$ is semisimple and has finite length. Thus $R/J(R) \in \mathcal{FL}.$ Hence, $\si(R/J(R))$ is the smallest element of  $\sip(\mathcal{FL})$ by Proposition \ref{cor:ss=fl}.
\end{proof}

%For a ring $R$, let $\mathcal{CSS}$ denote the class of (nonzero) cyclic semisimple $R$-modules.

%\begin{corollary} $R$ be a semilocal ring  and $n$ be the composition length of $R/J(R)$. Then $$|\mathcal{CSS}(R)| \leq 2^k$$
%where $k$ is the number of noninjective simple modules.
%\end{corollary}

\begin{corollary}\label{prop:smallest}  Let $R$ be  a $QF$ ring and $\mathcal{S}$ be a complete set of nonisomorphic  simple modules. Then 
\begin{enumerate}

\item $\si( \oplus_{S\in \mathcal{S}} S)= \inj$. 

\item $\si (\oplus_{S\in \mathcal{S}} S)$ is the smallest element of  $\sip(\mathcal{CP}).$

\end{enumerate}
\end{corollary}

\begin{proof}  $(1)$ $\oplus_{S\in \mathcal{S}} S $ is indigent by \cite[Proposition 32]{ABE}. Thus $\si(\oplus_{S\in \mathcal{S}} S )=\inj$.

$(2)$ Since $R$ is QF, it is artinian and right si.e.r. Thus $\mathcal{C} \subseteq \mathcal{FL}.$ Now, the proof follows by Proposition \ref{semilocal}. 
\end{proof}

Over a QF-ring every  finitely generated right module has finite length. 

\begin{corollary}\label{QF-inj} Let $R$ be a QF-ring. Then, 

$$\inj=\bigcap_{N \in \mathcal{S}} \si (N)=\bigcap_{N \in \mathcal{FL}} \si (N)=\bigcap_{N \in \mathcal{C}} \si (N)=\bigcap_{N \in \mathcal{FG}}\si (N).$$
\end{corollary} 

\begin{proof} By Proposition \ref{cor:ss=fl},  Proposition \ref{prop:C=FG} and Corollary \ref{prop:smallest}. 
\end{proof}

\begin{corollary}\label{QF-proj} Let $R$ be a QF-ring. Then, 

   $$\proj=\bigcap_{N \in \mathcal{S}} \spp (N)=\bigcap_{N \in \mathcal{FL}} \spp (N)=\bigcap_{N \in \mathcal{C}} \spp (N)=\bigcap_{N \in \mathcal{FG}}\spp (N).$$

\end{corollary} 

\begin{proof} By \cite[Theorem 2.10]{BD}, the module $\oplus_{N \in \mathcal{S}}N$ is p-indigent. That is, $\proj=\bigcap_{N \in \mathcal{S}} \spp (N)$. Thus the proof follows by Proposition \ref{prop:sp.e.r}.
\end{proof}

Since $\inj = \proj$ over  QF-rings, from Corollaries \ref{QF-inj} and  \ref{QF-proj} we immediately deduce the following.

\begin{corollary}\label{cor:QFmain} Let $R$ be a QF-ring. Then for each $\mathcal{A},\, \mathcal{B} \in \{\mathcal{S},\, \mathcal{C},\,\mathcal{FL},\,\mathcal{FG}\}$

$$\inj=\bigcap_{A \in \mathcal{A}}\si(A)= \bigcap_{B \in \mathcal{B}}\spp(B)=\proj $$ 

%$$\proj=\bigcap_{N \in \mathcal{S}} \spp (N)=\bigcap_{N \in \mathcal{FL}} \spp (N)=\bigcap_{N \in \mathcal{C}} \spp (N)=\bigcap_{N \in \mathcal{FG}}\spp (N).$$

\end{corollary}

%Artinian serial ring with $J^2=0$. Is $R$ Kasch? Such ring might be dual Kasch but not self injective. Thus $R$ is not a homomorphic image of an injective module. Otherwise $R$ would be self-injective. 

%Let $A$ be a right $R$-module of finite length, say $n$ and $A_0=0 \subset A_1  \subset  \cdots \subset A_n =A$ be a composition series of $A.$  The set of  simple factors $A_{i}/A_{i-1}$ for $i=1,\cdots, n$ is  called the set of composition factors of $A.$ As it is well-known these factor are independent from the choice of the composition series of $A.$

%\begin{proposition} Let $R$ be a QF-ring and $N$ be a  module of finite length with (nonisomorphic) composition factors $U_1,\,\cdots, U_k.$ Then $$\si(\oplus_{i=1}^k U_i) \subseteq \si(N).$$
%\end{proposition}

%\begin{proof} Let  $0 \subset N_1  \subset  \cdots \subset N_m =N$  be a composition series of $N.$ Suppose that $M \in \si(\oplus_{i=1}^k U_i)=\cap_{i=1} ^k \si(U_i).$
%Since $N_1 \simeq U_i$ and $N_2/N_1\simeq U_j$ for some $1\leq i,j \leq k,$ we have $M \in \si(N_1) \cap \si (N_2 /N_1).$ Then $M \in \si (N_2).$ Using induction on $m,$ we obtain $M \in \si (N),$ as desired. 
%\end{proof}

%\begin{corollary} Let $R$ be QF-ring with $n$ nonisomorphic noninjective simple right $R$-modules. Then $|\mathcal{CI}(R)| \geq 2^n.$
%\end{corollary}

%\begin{corollary} Let $R$ be a QF ring and $B=\{S_1,\cdots, S_m\}$ be nonisomorphic simple right $R$-modules. Then for any right $R$-module $N$ of finite length whose composition factors are contained in $B,$ we have $$\si(\oplus_{i=1}^m S_i) \subseteq \si(N).$$
%\end{corollary}

\begin{proposition}\label{prop:sierindigent} Let $R$ be right fully si.e.r ring and $B \in \Mod$. Then 
\bigskip

$B$ is indigent  if and only if for every $A \leq B,$ $\si(A) \cap \si(B/A)=\inj.$ 

\end{proposition}

\begin{proof} Suppose $B$ is indigent. Then $\si(B)=\inj.$ By the QF-ring assumption, every module is si.e.r. Thus $\si(A) \cap \si(B/A)\subseteq \si (B)=\inj.$ Hence $\si(A) \cap \si(B/A)=\inj.$ 

Conversely, assume that $\si(A) \cap \si(B/A)=\inj$ for each $A \leq B.$ For $A=0,$ we have $\si(0)=\Mod$ and so $\si(B)=\inj.$ Therefore, $B$ is indigent.
\end{proof}

\begin{proposition} Let $R$ be right fully sp.e.r ring and $B\in \Mod$. Then 

\bigskip
\item $B$ is p-indigent  if and only if for every $A \leq B,$ $\spp(A) \cap \spp(B/A)=\proj.$ 

\end{proposition}

\begin{proof} Dual to proof of Proposition \ref{prop:sierindigent}. 
\end{proof}

Over a QF-ring, indigent and p-indigent modules coincide. Hence the following is clear.

\begin{corollary}\label{QF:cor} Let $R$ be a QF-ring and $B\in \Mod$. Then, 
\bigskip

\item $B$ is indigent (=p-indigent)  if and only if for every $A \leq B,$ $\si(A) \cap \si(B/A)=\inj (=\proj).$

\end{corollary}

Recall that, a module $M$ is called $FG$-projective (resp. $C$-projective) in case for any epimorphism $f: N \to M$ and any homomorphism $g: F \to M$ with an $F$ finitely generated (resp. cyclic), there is an $h: F \to N$ such that $g=fh.$   By \cite[Lemma 1]{DKS}, $M$ is $FG$-projective (resp. $C$-projective) if and only if $M \in \spp(N)$ for each $N \in \mathcal{FG}$ (resp. $N \in \mathcal{C}$).

In \cite{DO},  a module $M$ is called $FG$-injective, if for every finitely generated module $K,$ every homomorphism $f: M \to K$ factors through an injective module. $M$ is $FG$-injective if and only if $M \in \si (N)$ for every finitely generated module $N$ (see, \cite[Proposition 2.3.]{DO}). To be consistent with the terminology, let us call $M$ $C$-injective in case any homomorphism $f: M \to C$ with a cyclic $C$ factors through an injective module. Similar arguments as in  \cite[Proposition 2.3.]{DO}  are applied to show that $M$ is $C$-injective if and only if  $M \in \si (N)$ for each $N \in \mathcal{C}.$

\begin{proposition}\label{FGC} The following hold for a module $M.$
\begin{enumerate}
\item If $M$ is si.e.r., then $M$ is $FG$-injective if and only if $M$ is $C$-injective.

\item If $M$ is sp.e.r., then $M$ is $FG$-projective if and only if $M$ is $C$-projective.
\end{enumerate}
\end{proposition}

\begin{proof} $(1)$  and $(2)$ follows by Proposition \ref{prop:C=FG} and Proposition \ref{prop:sp.e.r}, respectively.
\end{proof}

Now, the following directly follows by Corollary \ref{cor:QFmain} and Proposition \ref{FGC}.

\begin{corollary} Over a QF-ring, the notions of being $FG$-injective, $C$-injective, $FG$-projective, $C$-projective and injective coincide.
\end{corollary}

\section{Rings whose modules of finite length are homomorphic image of injective modules}

In this section, we explore a class of rings properly lies between QF-rings and dual Kasch rings. Namely, those rings  for which right modules of finite (composition) length are homomorphic images of injective modules. To this end, we consider the following property for a ring $R:$

%Recall that,  a ring $R$ is said to be right dual Kasch if every simple right $R$-module is a homomorphic image of an injective right $R$-module. Now we propose a stronger notion than right dual Kasch.  $R$ is said to be super dual Kasch if every right $R$-module with finite (composition) length is a homomorphic image of an injective right $R$-module. This new notion properly generalize dual Kasch rings as we see from the following example.

\bigskip

$(Q):$ Every right R-module of finite length is a homomorphic image of an injective module. 

\bigskip
%Right self-injective and  right $V$-rings are examples of rings  satisfy $(Q).$ 

Following \cite{meric},   $R$ is said to satisfy $(P)$ if every proper cyclic module is a homomorphic image of an injective module.

We have the following implications among the aforementioned classes of rings and non of which is reversible.

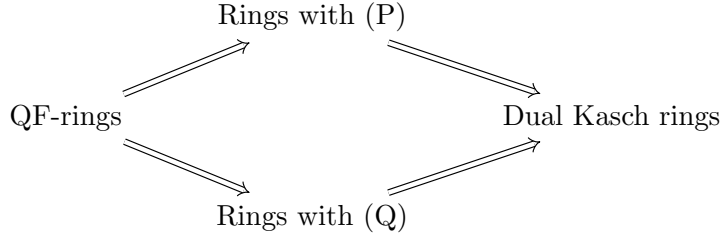
\[
\begin{tikzcd}
& \text{Rings with  (P)} \arrow[Rightarrow]{dr} & \\
\text{QF-rings} \arrow[Rightarrow]{ur} \arrow[Rightarrow]{dr} & & \text{Dual Kasch rings} \\
& \text{Rings with  (Q)} \arrow[Rightarrow]{ur} &
\end{tikzcd}
\]

%\begin{example}\label{example} Consider the ring $$R=\{\left(\begin{array}{cc}a & (b,c) \\0 & a\end{array}\right) \mid a,b,c \in \Z_2\}.$$Then $R$ is a commutative (finite) artinian ring with socle $soc(R)=\left(\begin{array}{cc}0 & \Z_2 \times \Z_ 2 \\0 & 0\end{array}\right).$ Then $R$ is dual Kasch i.e. every simple $R$-module is an epimorphic image of an injective module by \cite{BL}. \\
%On the other hand, since $soc(R)$ is not simple, $R$ is not a QF ring by \cite{Lammodules}. In particular, $R$ is not self-injective. Therefore $R$ is not an epimorphic image of an injective $R$-module, although $R$ has finite composition length. Hence $R$ is not super dual Kasch.
%\end{example}

%The following example shows that, the assumption of projectivity on $E(R)$ in Proposition \ref{prop:E(R) projective} is not superflous.

%The above example suggests a closer look to the rings whose right modules of finite length are homomorphic images of injective right modules. Let us call such ring right %\it{super dual Kasch.} Here are some examples of super dual Kasch rings.

 We begin with the following characterization of the rings satisfying $(Q)$ in terms of subinjectivity domain of modules with finite length.

\begin{proposition} The following statements are equivalent for a ring $R.$

\begin{enumerate}
\item $R$ satisfies $(Q)$.

\bigskip

\item $R \in \bigcap _{A \in  \mathcal{FL} } \si (A).$

\bigskip

\item $\proj \subseteq \bigcap _{A \in \mathcal{FL}  }\si (A).$

\end{enumerate}
\end{proposition}

\begin{proof} $(1) \Rightarrow (3)$ Suppose $R$ is a ring satifying (Q) and let $A \in \mathcal{FL}.$ Then there is an injective right $R$-module $E$ and an epimorphism $g: E \to A.$ Let $P \in \proj$ and $f: P \to A$ be a homomorphism. Suppose $P\leq K.$ Then, by projectivity of $P$, there is a homomorphism $h:P \to E$ such that $gh=f.$ Now, by injectivity of $E$, there is a homomorphism $t: K \to E$ such that $ti=h,$ where $i: P \to K$ is the inclusion map. Thus $f=gh=gti,$ that is, the map $gh$ extends $f.$ Hence $P \in \si (A),$ and so $(3)$ follows.

$(3) \Rightarrow (2)$ is clear.

$(2) \Rightarrow (1)$ Let $A\in \mathcal{FL}.$  As $A$ is finitely generated, there is an epimorphism $f : R^k \to A$ for some $k \in \Z^+ .$ Since $R \in \si(A)$ and  $\si (A)$ is closed under finite direct sums, $R^k \in  \si(A).$ Set $E=E(R^k).$ Then there is a homomorphism $g: E \to A$ such that $f=gi,$ where $i: R^k \to E$ is the inclusion map. Now $g$ is an epimorphism, because $f$ is so. Hence $A$ is an epimorphic image of $E,$ and  $(1)$ follows.
\end{proof}

\begin{lemma}\label{lem:ArtinQF}  A right artinian ring $R$  satisfies $(Q)$ if and only if $R$ is QF. 
\end{lemma}

\begin{proof} Sufficiency is clear. Suppose $R$ satisfies $(Q)$. Since $R$ is right artinian, $R$ has finite length. Thus $R$ is an epimorphic image of an injective right $R$-module, say $E$, by the assumption. Then $R$ is a direct summand of $E,$ and so $R$ is right self-injective.  Hence $R$ is a QF ring.
\end{proof}

\begin{lemma}\label{lem:VringSDK}  Any right $V$-ring  satisfies (Q).
\end{lemma}

\begin{proof} Let $R$ be aright $V$-ring and $A$ be a right $R$-module of finite length. Then $soc(A)$ is finitely generated and an essential submodule of $A.$ Thus $soc(A)$ is injective by the right $V$-ring assumption. Whence $soc(A)$ is a direct summand of $A$ and essentiality of $soc(A)$ in $A,$ implies $soc(A)=A.$  Hence $A$ is injective, and so $R$ is satisfies (Q).
\end{proof}

\begin{lemma}  For a right hereditary ring $R$, the following statements are equivalent. 

\begin{enumerate}
\item[(1)] $R$ satisfies $(Q)$.

\bigskip

\item[(2)] $R$ is right dual Kasch.

\bigskip

\item[(3)] $R$ is right $V$-ring.

\end{enumerate}
\end{lemma}

\begin{proof} 
$(1) \Rightarrow (2)$ is clear.  $(2) \Rightarrow (3)$ by \cite[Proposition 2.9]{BL}.   $(3) \Rightarrow (1)$ by Lemma \ref{lem:VringSDK}.
%Suppose that $R$ is right hereditary, and let $S$ be a simple right $R$-module. Then $S$ is a homomorphic image of an injective right  $R$-module $E$. Since $R$ is right hereditary, every factor of $E$ is injective. Hence $S$ is injective and so $R$ is a right $V$-ring. This proves the necessity.

%Sufficiency is clear by Lemma \ref{lem:VringSDK}.
\end{proof}

%Which commutative Noetherian rings are super dual Kasch?

%\begin{proposition} A commutative noetherian ring $R$ is super dual Kasch if and only if $R$ is $QF.$
%\end{proposition}

%\begin{proof} Suppose $R$ is a super dual Kasch ring. Then $R$ is dual Kasch, and so $R$ is a Kasch ring by \cite[Theorem 4.3.]{BL}. Since $R$ is noetherian $soc(R)$ is finitely generated. Let $soc(R)=S_1\oplus \cdots \oplus S_n$, where $S_i$ is simple for each $i=,\cdots, n.$ 

%\end{proof}

%\begin{proposition} Let $R$ be a commutative (noetherian) ring with injective hull $E(R).$ Then $R$ is super dual Kasch if and only if every indecomposable $R$-module of finite length is a homomorphic image of $E(R).$
%\end{proposition}

%\begin{proof}

%\end{proof}

\begin{remark} Every commutative artinian ring is a Kasch ring and hence also a dual Kasch ring by \cite[Corollary 3.2]{ABE}. An artinian ring satisfies  $(Q)$ if and only if it is QF by Lemma \ref{lem:ArtinQF}. Therefore, any commutative artinian ring that is not QF provides an example of a dual Kasch ring that does not satisfy $(Q)$ (see, Example \ref{example}). 
\end{remark}

In the following result, we provide a condition under which dual Kasch rings satisfy  $(Q)$.

\begin{proposition}\label{prop:E(R) projective} Let $R$ be a ring with projective $E(R)$. Then $R$ is right dual Kasch if and only if  $R$ satisfies $(Q)$.
\end{proposition}

\begin{proof} To prove the necessity, assume that $R$ is right dual Kasch. Then $R \in \si(S)$ for each simple right $R$-module $S.$ Let $B$ be a nonzero right $R$-module of length $n.$ First we shall prove that $R \in \si (B)$  by induction on $n.$ If $n=1,$ then $B$ is simple, and so the claim follows by the assumption. Assume that $n>1$, and every nonzero module of composition length less than $n$ is a homomorphic image of an injective module. Let $A$ be a nonzero submodule of $B$ that have composition length $n-1.$ Then $B/A$ has composition length 1, and so $R \in \si(A) \cap \si(B/A)$ by the induction hypothesis. Thus $R \in \si (B)$ by  hypothesis and Proposition \ref{lem:injdm} . Since $\si(B)$ is closed under finite direct sums, $R^k \in \si (B)$ for each $k \in \Z ^+ .$ Let $f:R^m \to B$ be an epimorphism. Then, as $R^m \in \si (B),$ there is a homomorphism $g: E(R)^m \to B$ extending $f.$ Since $f$ is an epimorphism, $g$ is an epimorphism. Therefore, $B$ is an epimorphic image of the injective module $E(R)^m.$ This proves the necessity. Sufficiency is clear.
\end{proof}

Note that,  a right $R$-module $M$ has finite length if and only if $M$ is both artinian and noetherian. Since being artinian, noetherian and being injective for modules are Morita invariant properties, we have the following.

\begin{proposition} Property $(Q)$ for rings is a Morita invariant property.
\end{proposition}

\begin{proof}
Let $R$ and $S$ be rings, $F:\Mod \rightarrow \Mods$ a Morita equivalence with inverse $G:\Mods\rightarrow \Mod$ such that $S$ is a ring satisfying $(Q)$. Let $A$   be a right $R$-module of finite length. Then $F(A)$  has finite length by \cite[Proposition 21.8(4)]{AF}.  Since $S$ satisy $(Q)$, there is an injective right $S$-module $E$ and an epimorphism $f: E \to F(A)$ in $\Mods.$ Then $G(E)$ is an injective right $R$-module by \cite[Proposition 21.6 (2)]{AF} and $G(f): G(E) \to G(F(A)) \simeq A$ is an epimorphism by \cite[Proposition 21.3]{AF}. Thus $A$ is an epimorphic image of the injective right $R$-module $G(E)$. This shows that $R$ a ring satisfying $(Q)$.  
\end{proof}

Since factor rings for QF-rings need not be QF, the rings that satisfy $(Q)$  need not be closed under factor rings.  

\begin{example}  Let $R=k[x,y]/(x^2,\,y^2)$, where $k$ is a field. Then, by \cite[Exercise 15.5]{Lammodules} $R$ is a commutative 4-dimensional Frobenius $k$-algebra. Thus $R$ is a ring satisfying $(Q)$. On the other hand, for the ideal $I=(x^2, y^2)/(x^2, xy, y^2)$, the factor ring $R/I \simeq k[x,y]/(x^2,\, xy, \,y^2)$ is artinian but not self-injective. Therefore $R/I$ is not QF, and so $R/I$ does not satisfy $(Q)$ by Lemma \ref{lem:ArtinQF}. 
\end{example}

A ring $R$ is said to be super $QF$ if and only if every factor ring of $R$ is QF (see, \cite{SS}). 

\begin{proposition}\label{prop:superDK=superQF} Let $R$ be a right artinian ring. Every factor ring of $R$ satisfies $(Q)$ if and only if $R$ is super QF.
\end{proposition}

\begin{proof} Factor rings of artinian rings are artininian. Thus the proof follows by  Lemma \ref{lem:ArtinQF}.
\end{proof}

\begin{proposition} Let $R$ be a commutative noetherian ring. Every factor ring of $R$ satisfies $(Q)$ if and only if $R$ is super QF.
\end{proposition}

\begin{proof} Sufficiency is clear.  To prove the necessity, assume that every factor ring of $R$ satisfy $(Q)$.  By Proposition \ref{prop:superDK=superQF} , it is enough to show that $R$ is artinian. Let $s(R)$ be the sum of semiartinian ideals of $R$. Suppose $s(R) \neq R.$ Then $soc(R/s(R))=0,$ and by the assumption, $R/s(R)$ is a ring satisfying   $(Q)$. Thus $R/s(R)$ is a Kasch ring by \cite[Theorem 3.4.]{BL}. This implies that, $soc(R/s(R))\neq 0.$ Therefore we must have $s(R)=R.$ Thus being semiartian and noetherian, implies $R$ is artinian. Hence $R$ is super $QF$ by  Proposition \ref{prop:superDK=superQF}.
\end{proof}

\section*{Acknowledgement}

This work was supported by the Scientific and Technological Research Council of T\"{u}rkiye (TÜBİTAK) (Project number: 122F130).

\end{document}